\documentclass{amsart}

\usepackage{amsmath}
\usepackage{amssymb}
\usepackage{amsthm}
\usepackage{mathabx}
\usepackage{pgfplots}
\usepackage{tikz}
\usetikzlibrary{trees}
\usetikzlibrary{arrows}
\usepackage{hyperref}
\usepackage[T2A]{fontenc}
\usepackage[utf8]{inputenc}
\usepackage{bbm}
\usepackage{datetime}
\DeclareMathAlphabet{\mathbbold}{U}{BOONDOX-ds}{m}{n}

\DeclareMathOperator\supp{supp}
\DeclareMathOperator\Fol{Føl}
\DeclareMathOperator\Id{Id}

\newfont{\bbf}{msbm10 scaled\magstep1}

\def\arxivprefix{https://arxiv.org}

\newcounter{glob}[section]
\renewcommand\theglob{%
	\ifnum\arabic{section}=0\else\arabic{section}.\fi %
	\arabic{glob}}

\theoremstyle{plain}
\newtheorem{thm}[glob]{Theorem}
\newtheorem{lemma}[glob]{Lemma}
\newtheorem{cor}[glob]{Corollary}
\newtheorem{prop}[glob]{Proposition}
\newtheorem*{prop*}{Proposition}

\theoremstyle{definition}
\newtheorem{defi}[glob]{Definition}
\newtheorem{ex}[glob]{Example}

\theoremstyle{remark}
\newtheorem{remark}[glob]{Remark}
\newtheorem*{remark*}{Remark}

\def\Z{\mathbb{Z}}

\def\R{\mathbb{R}}

\def\N{\mathbb{N}}

\pgfplotsset{width=8cm}

\newdateformat{UKvardate}{%
	\THEDAY\ \monthname[\THEMONTH], \THEYEAR}
\UKvardate

\title{\bf Exact descriptions of Følner functions and sets on wreath products and Baumslag-Solitar groups}
\newcommand*{\theauthor}{Bogdan Stankov}
\author{\theauthor}
\address{Institut Camille Jordan\\Universit\'e Claude Bernard Lyon 1\\43 boulevard du 11 novembre 1918\\F-69622 Villeurbanne Cedex}
\email{bogdan.zl.stankov@gmail.com}
\date{\today}

\newcommand*{\thekeywords}{F{\o}lner function, F{\o}lner sets, lamplighter group, wreath products, Baumslag-Solitar groups, Coulhon and Saloff-Coste inequality, growth function}

\thanks{
I would like to thank the D\'epartement de math\'ematiques et applications laboratory at \'Ecole normale sup\'erieure de Paris, where I wrote the first version of this paper.
My work there was supported by the ERC grant GroIsRan.
I would also like to thank the Institut Camille Jordan at Universit\'e Claude Bernard Lyon 1.
My work there is supported by LabEx MILYON.
}

\keywords{\thekeywords}

\makeatletter
\hypersetup{
pdfauthor=\theauthor,
pdftitle=\@title,
pdfsubject=\@title,
pdfkeywords=\thekeywords,
pdflang={en-US}
}
\makeatother

\begin{document}
\maketitle
 
\begin{abstract}
We calculate the exact values of the Følner function $\Fol$ of the lamplighter group $\Z\wr\Z/2\Z$ for the standard generating set.
More generally, for any finite group $D$ and $n\geq|D|$, we obtain the exact value of $\Fol(n)$ on the wreath product $\Z\wr D$, for a generating set induced by a generator on $\Z$ and the entire group being taken as generators for $D$.
We also describe the Følner sets that give rise to it.
Følner functions encode the isoperimetric properties of amenable groups and have previously been studied up to asymptotic equivalence (that is to say, independently of the choice of finite generating set).
What is more, we prove an isoperimetric result concerning the edge boundary on the Baumslag-Solitar group $BS(1,2)$ for the standard generating set.
\end{abstract}

\section{Introduction}
\thispagestyle{empty}

One equivalent characterization of the amenability of an infinite group $G$, called the \textit{F{\o}lner condition}, is that the isoperimetric constant (also known as Cheeger constant) of its Cayley graph should be $0$ (for some finite generating set $S$).
That constant is defined as the infimum of $\frac{|\partial F|}{|F|}$ over all finite sets $F\subset G$ with $|F|\leq\frac{1}{2}|G|$.
As the quotient cannot reach $0$, amenability of infinite groups is therefore characterized by the existence of a sequence of sets $F_n$ such that $\frac{|\partial  F_n|}{|F_n|}$ converges towards $0$, also known as a \textit{F{\o}lner sequence}.
One natural direction for studying the possible F{\o}lner sequences on a given group is to ask how small the sets can be.
We consider the F{\o}lner function.
It has classically been defined using the inner boundary:
\begin{equation}\label{defdin}
	\partial_{in}F=\left\{g\in F:\exists s\in S\bigcup S^{-1}:gs\notin F\right\}.
\end{equation}
\begin{defi}\label{foldef}
	The \textit{F{\o}lner function} $\Fol$ (or $\Fol_S$; or $\Fol_{G,S}$) of a group $G$ with a given finite generating set $S$ is defined on $\N$ by
	$$\Fol(n)=\min\left(|F|:F\subset G,\frac{|\partial_{in}F|}{|F|}\leq\frac{1}{n}\right).$$
\end{defi}

Remark that $\Fol(1)=1$ and that the values of the function are finite if and only if $G$ is amenable.
Its values clearly depend on the choice of a generating set, but the functions arising from different generating sets (and more generally, functions arising from quasi-isometric spaces) are asymptotically equivalent.
Two functions are asymptotically equivalent if there are constants $A$ and $B$ such that $f(x/A)/B<g(X)<f(xA)B$.
Various articles have classified F{\o}lner functions up to asymptotic equivalence, see for example~\cite{BrieusselZheng,Erschler2003,Erschler2017,Gromov2009,Osserman1978,Pittet1995,Pittet2003,saloffcostezheng}.
In this paper, we will determine its exact values for one of the most basic examples of groups with exponential growth (we state our results in Section~\ref{statesect}).

The classical isoperimetric theorem states that among domains of given volume in $\R^n$, the minimal surface area is obtained on a ball (see survey by Osserman~\cite[Section~2]{Osserman1978}).
The fact that if a minimum exists, it is realized only on the ball is obtained (in $\R^2$) by Steiner in the XIX\textsuperscript{th} century, using what is now called Steiner symmetrization (see Hehl~\cite{hehl13}, Hopf~\cite{hopf40}, Froehlich~\cite{froehlich09}).
The existence of a minimum is obtained, in $\R^3$, by Schwarz~\cite{Schwarz1884}.
As $\Z^n$ is quasi-isometric to $\R^n$, this is also a first isoperimetric result for discrete groups.
Varopoulos~\cite{Varopoulis1985} shows more generally an isoperimetric inequality for direct products.
Pansu~\cite{pansu83} (see also~\cite{pansu82}) obtains one for the Heisenberg group $H_3$.
One central result is the Coulhon and Saloff-Coste inequality~\cite{Coulhon1993}:

\begin{thm}[Coulhon and Saloff-Coste inequality]\label{csc}
Consider a group $G$ generated by a finite set $S$ and let $\phi(\lambda)=\min(n:V(n)>\lambda)$.
Then for all finite sets $F$:

$$\frac{|\partial_{in}F|}{|F|}\geq\frac{1}{8|S|\phi(2|F|)}.$$
\end{thm}
The multiplicative constants can improved (see Gábor Pete~\cite[Theorem~5.11]{pete}, Bruno Luiz Santos Correia~\cite{csc2020}):
\begin{equation}\label{csceq}
\frac{|\partial_{in}F|}{|F|}\geq\frac{1}{2\phi(2|F|)}.
\end{equation}
The result of Santos Correia is also announced for finite groups for $|F|\leq\frac{1}{2}|G|$.
Santos Correia and Troyanov~\cite{csc2021} show:
\begin{equation}\label{csceqlam}
\frac{|\partial_{in}F|}{|F|}\geq\left(1-\frac{1}{\lambda}\right)\frac{1}{\phi(\lambda|F|)}
\end{equation}
for $1<\lambda\leq\frac{|G|}{|S|}$ (in particular, for arbitrarily large $\lambda$ if $G$ is infinite).
The result is replicated in~\cite{csccollab} without an upper bound on $\lambda$.

The Coulhon and Saloff-Coste inequality (Theorem~\ref{csc}) implies in particular that for a group with exponential growth, the F{\o}lner function must also grow at least exponentially.
It can be obtained then that it is exactly exponential if one can describe F{\o}lner sets with exponential growth.
One simple example is the lamplighter group $\Z\wr\Z/2\Z$ with the standard generating set (\ref{defst}).
Similarly, it is known that the F{\o}lner functions of groups with polynomial growth are polynomial (see for example~\cite[Section~I.4.C]{woess2000random}).
Another inequality on group isoperimetry is given by Żuk~\cite{Zuk2000}.
Vershik~\cite{vershik1973countable} asks if F{\o}lner function can be super-exponential, initiating the study of F{\o}lner functions.
He suggests studying the wreath product $\Z\wr\Z$ as a possible example.
Pittet~\cite{Pittet1995} shows that the F{\o}lner functions of polycyclic groups are at most exponential (and are therefore exponential for polycyclic groups with exponential growth).
This is true more generally for solvable groups of finite Prüfer rank, see~\cite{Pittet2003} and~\cite{Kropholler2020}.
The first example of a group with super-exponential F{\o}lner function is obtained by Pittet and Saloff-Coste~\cite{PittetSaloffCoste} for $\Z^d\wr\Z/2\Z$ with $d\geq3$.
Later the F{\o}lner functions of wreath products with certain regularity conditions are described by Erschler~\cite{Erschler2003} up to asymptotic equivalence.
Specifically, say that a function $f$ verifies property $(*)$ if for all $C>0$ there is $k>0$ such that $f(kn)>Cf(n)$.
The result of \cite{Erschler2003} than states that if the F{\o}lner function of a group $A$ verifies property $(*)$ (for some fixed generating set), then for any non-trivial group $B$, the F{\o}lner function of $A\wr B$ is $\Fol_{A\wr B}(n)=\Fol_B(n)^{\Fol_A(n)}$.

Other examples with know F{\o}lner functions have been presented by Gromov~\cite[Section~8.2,Remark~(b)]{Gromov2009} for all functions with sufficiently fast growing derivatives.
Saloff-Coste and Zheng~\cite{saloffcostezheng} provide upper and lower bounds for it on, among others, "bubble" groups and cyclic Neumann-Segal groups, and those two bounds are asymptotically equivalent under certain conditions.
Recently, Brieussel and Zheng~\cite{BrieusselZheng} show that for any function $g$ that can be written as the inverse function of $x/f(x)$ for some non-decreasing $f$ with $f(1)=1$ and $x/f(x)$ also non-decreasing, there is a group the F{\o}lner function of which is asymptotically equivalent to $\exp(g(n))$.
Erschler and Zheng~\cite{Erschler2017} obtain examples for a class of super-exponential functions under $\exp(n^2)$ with weaker regularity conditions.
Specifically, for any $d$ and any non-decreasing $\tau$ such that $\tau(n)\leq n^d$, there is a group $G$ and a constant $C$ such that
\begin{equation}\label{annatyani}
	Cn\exp(n+\tau(n))\geq\Fol_G(n)\geq\exp(\frac{1}{C}(n+\tau(n/C))).
\end{equation}
The left-hand side of this inequality is always asymptotically equivalent to $\exp(n+\tau(n))$, and it suffices therefore that the right-hand side be asymptotically equivalent to that function to have a description of the F{\o}lner function of $G$.
Notice in particular that if $\tau$ verifies condition $(*)$, this is verified.
Remark that the conditions we mentioned only consider functions at least as large as $\exp(n)$; it is an open question whether a F{\o}lner function can have intermediate growth (see Grigorchuk~\cite[Conjecture~5(ii)]{grigsurvey}).
A negative answer would imply the Growth Gap Conjecture~\cite[Conjecture~2]{grigsurvey}, which conjectures that the volume growth function must be either polynomial or at least as fast as $\exp(\sqrt{n})$.
Those conjectures also have weak versions, which are equivalent to each other (see discussion after Conjecture~6 in~\cite{grigsurvey}).

\section{Statement of results}\label{statesect}

In this paper, we obtain exact values of the F{\o}lner function for a generating set $S_D=\{t\}\bigcup\{\delta_i;i\in D\setminus\{Id_D\}\}$ (see~(\ref{defst})) on the wreath product $\Z\wr D$ (see Definition~\ref{deflamp}).
Remark that for the lamplighter group $\Z\wr\Z/2\Z$, this is the standard generating set.
Remark also that the group structure of $D$ is irrelevant for the study of F{\o}lner sets with regards to $S_D$.

\begin{thm}\label{thmmain}
For any finite group $D$ and any $n\geq|D|$, the value of the F{\o}lner function of the wreath product $\Z\wr D$ for the generating set $S_D=\{t\}\bigcup\{\delta_i;i\in D\setminus\{Id_D\}\}$ at $n$ is: 
$$\Fol_D(n)=2n|D|^{2(n-1)}.$$
\end{thm}

We also describe the sets that give rise to this function.
Specifically, we obtain that the standard F{\o}lner sets $F_n=\{(k,f):k\in[\![1,n]\!],\supp(f)\subset[\![1,n]\!]\}$ are optimal (see Definition~\ref{optimal}) for the outer and edge boundary (see Section~\ref{prel1}).
We then show that by Lemma~\ref{equiv}, $F_n\bigcup\partial_{out}F_n=\{(k,f):k\in[\![0,n+1]\!],\supp(f)\subset[\![1,n]\!]\}$ is optimal for the inner boundary, from which Theorem~\ref{thmmain} follows:

\begin{thm}\label{thmlamp}
Consider the wreath product $\Z\wr D$ with the generating set $S_D=\{t\}\bigcup\{\delta_i;i\in D\setminus\{Id_D\}\}$.
Let $F_n=\{(k,f):k\in[\![1,n]\!],\supp(f)\subset[\![1,n]\!]\}$.
\begin{enumerate}
	\item For any $n\geq2(|D|-1)$ and any $F\subset\Z\wr D$ such that $|F|\leq|F_n|$, we have
	$$\frac{|\partial_{edge}F|}{|F|}\geq\frac{|\partial_{out}F|}{|F|}\geq\frac{|\partial_{out}F_n|}{|F_n|}=\frac{|\partial_{edge}F_n|}{|F_n|},$$
	and if $|F|<|F_n|$, the inequality $\frac{|\partial_{out}F|}{|F|}>\frac{|\partial_{out}F_n|}{|F_n|}$ is strict,
	\item From point $(1)$ it follows that for any $n\geq2(|D|-1)$ and any $F\subset\Z\wr D$ such that $|F|\leq|F_n\bigcup\partial_{out}F_n|$, we have
	$$\frac{|\partial_{in}F|}{|F|}\geq\frac{|\partial_{in}(F_n\bigcup\partial_{out}F_n)|}{|F_n\bigcup\partial_{out}F_n|},$$
	and if $|F|<|F_n\bigcup\partial_{out}F_n|$, the inequality is strict.
\end{enumerate}
Furthermore, the sets that give equality are unique up to translation.
\end{thm}

We can substitute those values in the Coulhon and Saloff-Coste inequality (Theorem~\ref{csc}) in order to study the multiplicative constant.
As in~\cite{csccollab}, we define
\begin{defi}\label{constq}
For a group $G$ and a generating set $S$, denote
$$C_{G,S}=\sup\left\{c\geq0:\exists\alpha\geq0\mbox{ such that }\forall F\subset G,\frac{|\partial_{in}F|}{|F|}\geq c\frac{1}{\phi((1+\alpha)|F|)}\right\},$$
where $F$ is assumed to be finite and non-empty.
\end{defi}
The original inequality obtains that for all $G,S$, $C_{G,S}\geq\frac{1}{8|S|}$.
The results of~\cite[Theorem~5.11]{pete} and~\cite{csc2020} that we cited as Equation~\ref{csceq} give a lower bound of $\frac{1}{2}$.
Equation~\ref{csceqlam} from~\cite{csc2021} further implies that $C_{G,S}\geq1$ for all $G,S$.

In \cite{csccollab}, it is shown that for groups of exponential growth:
$$C_{G,S}=\frac{\liminf\frac{\ln\Fol(n)}{n}}{\lim\frac{\ln V(n)}{n}}.$$
\begin{prop}\label{const}
The wreath product $\Z\wr D$ verifies

$$C_{\Z\wr D,S_D}=\frac{\lim\frac{\ln\Fol_{\Z\wr D}(n)}{n}}{\lim\frac{\ln V_{\Z\wr D}(n)}{n}}=\frac{2\ln|D|}{\ln(\frac{1}{2}(1+\sqrt{4|D|-3}))}\geq\frac{2\ln2}{\ln(\frac{1}{2}(1+\sqrt{5}))}\approx2,88$$
for the generating set $S_D=\{t\}\bigcup\{\delta_i;i\in D\setminus\{Id_D\}\}$. 
\end{prop}

\begin{remark}\label{const2}
For the switch-walk-switch generating set $S_{sws}=\{t,\delta_1,t\delta_1,\delta_1 t,\delta_1 t\delta_1\}$ on the lamplighter group $\Z\wr\Z/2\Z$, we have

$$C_{\Z\wr\Z/2\Z,S_{sws}}=\frac{\liminf\frac{\ln\Fol_{sws}(n)}{n}}{\lim\frac{\ln V_{sws}(n)}{n}}\leq2.$$
\end{remark}

Another direction that can be considered once one has exact evaluations of a F{\o}lner function is studying the power series $\sum_n\Fol(n)x^n$.
The equivalent series have been studied for volume growth (see Grigorchuk-de la Harpe~\cite[Section~(4)]{Grigorchuk1997}).
One central question that a lot of authors have considered is the rationality of those series as a function.
For the example shown here, the power series of the F{\o}lner function is a rational function
$$\sum_{n\in\N}\Fol(n)x^n=\frac{2x}{(1-|D|^2x)^2}-x+x^2P_D(x),$$
where $P_D$ is a polynomial with $\deg{P_D}\leq|D|-3$.

We also obtain results for the Baumslag-Solitar group $BS(1,2)$ (see Definition~\ref{defbs}), however only with respect to the edge boundary.
Taking the notation from the definition, its standard sets are defined the same way as in the lamplighter group.
\begin{thm}\label{bsthm}
Consider the Baumslag-Solitar group $BS(1,2)$ with the standard generating set.
Then for any $n\geq2$ and any $F\subset BS(1,2)$ such that $|F|\leq|F_n|$, we have $\frac{|\partial_{edge}F|}{|F|}\geq\frac{|\partial_{edge}F_n|}{|F_n|}$ (where $F_n$ are the standard F{\o}lner sets), and if $|F|<|F_n|$, the inequality is strict.
\end{thm}
This result is not always true for $BS(1,p)$ for larger $p$, and we provide a counter example in Example~\ref{exbsp}.
However this counter example uses that $p$ is significant when compared to the length of the interval defining the standard set, and it is possible that for $BS(1,p)$ as well, standard sets are optimal above a certain size.

We present more detailed definitions in the next section.
In Section~\ref{prelim2}, we present associated graphs, which are the main tool of the proof, and prove some general results.
In particular, we show Lemma~\ref{equiv}, which will be used to obtain that part $(2)$ of Theorem~\ref{thmlamp} follows from part $(1)$.
In Section~\ref{mainsection}, we prove Theorem~\ref{thmlamp}.
In Section~\ref{csc-const-sect}, we prove Proposition~\ref{const} and Remark~\ref{const2}.
Finally, in Section~\ref{bssect}, we prove Theorem~\ref{bsthm} and Example~\ref{exbsp}.

\section{Preliminaries}\label{prel1}

The concept of amenability finds its origins in a 1924 result by Banach and Tarski~\cite{Banach-Tarski-original}, where they decompose a solid ball in $\R^3$ into five pieces, and reassemble them into two balls using rotations.
That is now called the Banach-Tarski paradox.
The proof makes use of the fact that the group of rotations of $\R^3$ admits a free subgroup.
Von Neumann~\cite{Neumann1929} considers it as a group property and introduces the concept of amenable groups.
Nowadays, there are multiple different characterizations of amenability; see books by Greenleaf~\cite{greenleaf} and Tomkowicz-Wagon~\cite{Banach-Tarski}, or an article by Ceccherini-Silberstein-Grigorchuk-la~Harpe~\cite{MR1721355}, or a recent survey by Bartholdi~\cite{bartholdi}.

\begin{defi}[F{\o}lner criterion]\label{folamdef}
	A group $G$ is amenable if and only if for every finite set $S\subset G$ and every $\varepsilon>0$ there exists a set $F$ such that
	
	$$|F\Delta F.S|\leq\varepsilon|F|.$$
\end{defi}

If $G$ is finitely generated, it suffices to consider a single generating set $S$ instead of all finite sets.
We can also apply Definition~\ref{folamdef} for $S\bigcup S^{-1}\bigcup\{\Id\}$.
Then $|F\Delta(S\bigcup S^{-1}\bigcup\{\Id\}).F|$ is the set of vertices in the Cayley graph of $G$ that are at a distance exactly $1$ from $F$.
We denote that the outer boundary $\partial_{out}F$.
Then the condition can be written as $\frac{|\partial_{out}F_n|}{|F_n|}\leq\varepsilon$, or in other words that the  infimum of those quotients should be $0$.
Recall the definition (\ref{defdin}) of the inner boundary.
Finally, we consider $\partial_{edge}F$ to be the set of edges between $F$ and its complement.
Remark that while those values can differ, whether the infimum of $\frac{|\partial F|}{|F|}$ is $0$ or not does not depend on which boundary we consider.

For groups of subexponential growth, for every $\varepsilon$, there is some $n$ such that the ball around the identity of radius $n$ is a corresponding F{\o}lner set.
Note that to obtain a F{\o}lner sequence from this, one needs to consider a subsequence of the sequence of balls of radius $n$.
It is an open question whether in every group of subexponential growth, all balls form a F{\o}lner sequence.
For groups of exponential growth, it is generally not sufficient to consider balls, and it is an open question whether there exists any group of exponential growth where some subsequence of balls forms a F{\o}lner sequence (see for example Tessera~\cite[Question~15]{Tessera2007}).

For two groups $A$ and $B$ and a function $f\in B^A$, denote
$$\supp(f)=\{a\in A:f(a)\neq\Id_B\}.$$
Let $B^{(A)}$ be the set of functions from $A$ onto $B$ with finite support.
\begin{defi}\label{deflamp}
	The (restricted) wreath product $A\wr B$ is the semidirect product $A\ltimes B^{(A)}$ where $A$ acts on $B^{(A)}$ by translation.
\end{defi}
We can write the elements as $(a,f)$ with $a\in A$ and $f\in B^{(A)}$.
The group law is then $(a,f)(a',f')=(aa',x\mapsto f(x)f'(a^{-1}x))$.

Given generating sets $S$ and $S'$ on $A$ and $B$ respectively, we can define a standard generating set on $A\wr B$.
It consists of the elements of the form $(s,\mathbb{Id_B})$ for $s\in S$ (where $\mathbb{Id_B}=\Id_B$ for all $x\in A$), as well as $(\Id_A,\delta_{\Id_A}^{s'})$ for $s'\in S'$ where $\delta_{\Id_A}^{s'}(\Id_A)=s'$ and $\delta_{\Id_A}^{s'}(x)=\Id_B$ for all other $x$.
One can verify that $(a,f)(s,\mathbb{Id_B})=(as,f)$, and $(a,f)(\Id_A,\delta_{\Id_A}^{s'})=(a,f+\delta_a^{s'})$, or in other words the value of $f$ at the point $a$ is changed by $s'$.

Similarly, given F{\o}lner sets $F_A$ and $F_B$ on $A$ and $B$ respectively, one obtains standard F{\o}lner sets on $A\wr B$: 
$$F=\{(a,f):a\in F_A,\supp(f)\subset F_A,\forall x:f(x)\in F_B\}.$$
One can check that $\frac{|\partial_{out}F|}{|F|}=\frac{|\partial_{out}F_A|}{|F_A|}+\frac{|\partial_{out}F_B|}{|F_B|}$.

We will focus on the wreath product $\Z\wr D$.
For $\Z$ we will consider the standard generator, while for $D$ we will take the entire group (except the identity) as generators:

\begin{equation}\label{defst}
S_D=\{t\}\bigcup\{\delta_i:i\in D\setminus\{Id_D\}\}\mbox{ where }t=(1,\mathbbold{0})\mbox{ and }\delta_i=(0,\delta^i_0).
\end{equation}

The Baumslag-Solitar groups are defined as follows:
\begin{defi}\label{defbs}
	The Baumslag-Solitar group $BS(m,n)$ is the two-generator group given by the presentation $\langle a,b:a^{-1}b^ma=b^n\rangle$.
\end{defi}
The standard generating set is $\{a,b\}$.

We will focus on the groups $BS(1,p)$.
That group is isomorphic to the group generated by $x\mapsto px$ and $x\mapsto x+1$ (by mapping $a^{-1}$ and $b$ to them respectively).
By abuse of notation, we will also denote the images of $a$ and $b$ with the same letters.
In that group, any element can be written as $x\mapsto p^nx+f$ with $n\in\Z$ and $f\in\Z[\frac{1}{p}]$.
We then have $(x\mapsto p^nx+f)a=x\mapsto p^{n-1}x+f$ and $(x\mapsto p^nx+f)b=x\mapsto p^nx+(f+p^n)$.

Remark that the subgroup $N=\{x\mapsto x+f:f\in\Z[\frac{1}{p}]\}$ is normal.
Indeed, we have $(x\mapsto p^{-n}(x-f))\circ(x\mapsto p^nx+f)=\Id$ and
$$(x\mapsto p^{-n}(x-f'))\circ(x\mapsto x+f)\circ(x\mapsto p^nx+f')=x\mapsto x+p^nf.$$
Thus $BS(1,p)$ is isomorphic to the semidirect product $\Z\ltimes\Z[\frac{1}{p}]$ defined by the action $n.f=p^nf$.
We therefore write the element $x\mapsto p^nx+f$ as $(n,f)$.
The standard F{\o}lner sets are then expressed in the same way as for wreath products.
In other words:
$$F_n=\{(k,f):k\in[\![0,n-1]\!],f\in\Z,0\leq f<p^n\}.$$

\section{Main concepts of the proof}\label{prelim2}

\begin{defi}\label{optimal}
	We will call a set $F$ in a group $G$ \textit{optimal} with respect to the inner (respectively outer, edge) boundary if for any $F'$ with $|F'|\leq|F|$, it is true that $\frac{|\partial_{in}F'|}{|F'|}\geq\frac{|\partial_{in}F|}{|F|}$ (respectively $\frac{|\partial_{out}F'|}{|F'|}\geq\frac{|\partial_{out}F|}{|F|}$, $\frac{|\partial_{edge}F'|}{|F'|}\geq\frac{|\partial_{edge}F|}{|F|}$), and if $|F'|<|F|$, the inequality is strict.
\end{defi}

\begin{lemma}\label{equiv}
If $F$ is optimal for the outer boundary, up to replacing $F$ with another optimal set of the same size, $F\bigcup\partial_{out}F$ is optimal for the inner boundary.
Furthermore, this describes all optimal sets of that size.
\end{lemma}

\begin{proof}
We will prove the large inequality.
The case for the strict inequality is equivalent.

Let $F$ be optimal for the outer boundary and consider $F'$ such that $|F'|\leq|F\bigcup\partial_{out}F|$ and the quotient of the inner boundary is smaller.
Without loss of generality, we can assume that $F'$ is optimal for the inner boundary.

Let $F''=F'\setminus\partial_{in}F'$.
Observe that $\partial_{out}F''\subset\partial_{in}F'$.
We first claim that $F'$ being optimal implies that $\partial_{out}F''=\partial_{in}F'$.
Indeed, $|F''\bigcup\partial_{out}F''|\leq|F'|$, and
$$\frac{|\partial_{in}(F''\bigcup\partial_{out}F'')|}{|F''\bigcup\partial_{out}F''|}\leq\frac{|\partial_{out}F''|}{|F''\bigcup\partial_{out}F''|}=\frac{\frac{|\partial_{out}F''|}{|F''|}}{1+\frac{|\partial_{out}F''|}{|F''|}},$$
while
$$\frac{|\partial_{in}F'|}{|F'|}=\frac{\frac{|\partial_{in}F'|}{|F''|}}{1+\frac{|\partial_{in}F'|}{|F''|}}.$$
As $\partial_{out}F''\subset\partial_{in}F'$ and $\frac{x}{x+1}$ is an increasing function in $\R_+$, the first quantity is smaller then the second, and by $F'$ being optimal, we have an equality.

We now consider cases for the size of $F''$.
If $|F''|<|F|$, then we can apply the assumption that $F$ is optimal and we get
$$\frac{|\partial_{out}F|}{|F|}\leq\frac{|\partial_{out}F''|}{|F''|}=\frac{|\partial_{in}F'|}{|F'\setminus\partial_{in}F'|}=\frac{\frac{|\partial_{in}F'|}{|F'|}}{1-\frac{|\partial_{in}F'|}{|F'|}}.$$
However, applying the initial assumption by which we chose $F'$ gives us the inverse inequality, and strict.

We are left with the case where $|F''|\geq|F|$.
Let $k=|F''|-|F|$ and remove any $k$ points from $F''$ to obtain $F'''$.
We obtain a set that is the same size as $F$ and has an outer boundary no larger than that of $F$.
It is therefore another optimal set of the same size, and by the optimality of $F'$ for the inner boundary, $F'''\bigcup\partial_{out}F'''=F'$, which concludes the proof.
\end{proof}

The central idea of this paper is to work on an associated graph structure which we define for semidirect products.
\begin{defi}\label{assgr}
Consider a semidirect product $G=H\ltimes N$.
Consider a generating set $S$ of $G$.
We define the \textit{associated graph} as the directed labeled graph $\Gamma=\Gamma_S$ with vertex set $V(\Gamma)=N$ and edge set 
	
$$\overrightarrow{E}(\Gamma)=\left\{\overrightarrow{(f_1,f_2)}:\exists s\in S,h_1,h_2\in H\mbox{ such that }h_1f_1s=h_2f_2\right\}.$$
With these notations, the edge $\overrightarrow{(f_1,f_2)}$ is labeled $s$.
\end{defi}
As mentioned, the two examples we will consider here are the lamplighter group $\Z\wr D$ and the Baumslag-Solitar group $BS(1,p)$.
In both examples we have $H=\Z$.
In the case of the wreath product we have $N=D^{(\Z)}$, and for $BS(1,p)$, $N$ is the set of $p$-adic numbers.

We define an associating function $\phi:G\rightarrow\mathcal{P}(\overrightarrow{E}(\Gamma))$ by 
$$\phi(hf)=\left\{\overrightarrow{(f,f')}:hfs=h'f'\mbox{ for some }h'\in H,s\in S\right\}.$$
\begin{defi}\label{asssubgr}
Consider a semidirect product $G=H\ltimes N$.
Let $F$ be a finite subset of $G$.
The \textit{associated subgraph} of $F$ is the (directed) subgraph of the associated graph $\Gamma$ made of the edges
$$\bigcup_{x\in F}\phi(x)$$
and all adjacent vertices.
\end{defi}
We will provide a bound for the boundary of a set based on a formula on the associated subgraph, and maximize the value of that formula over all subgraphs of $\Gamma$ no larger (in terms of number of edges) than the associated subgraph.

\section{The wreath product $\Z\wr D$}\label{mainsection}

In this section we provide the proof of Theorem~\ref{thmlamp}, the larger part being a proof of Theorem~\ref{thmlamp}$(1)$.
In other words, we show that for the wreath product with generating set $S_D$, the standard sets are optimal with respect to the outer boundary, and uniquely so up to translation.
We first prove an inequality relying the isoperimetry of a subset with values on the associated graph.
For the entirety of this section, we denote $d=|D|$.

\begin{lemma}\label{main}
Let $F$ be a finite set in $\Z\wr D$ with $\frac{|\partial_{out}F|}{|F|}\leq\frac{1}{d-1}$.
Let $\Gamma$ be the associated graph (see Definition~\ref{assgr}).
Then

$$\frac{|\partial_{out}F|}{|F|}\geq\min\left(\frac{2(d-1)|V(G)|}{|E(G)|}\mbox{ for }G\mbox{ subgraph of }\Gamma\mbox{ with }|E(G)|\leq(d-1)|F|\right).$$
\end{lemma}

\begin{proof}
As mentioned below the definition of associated graph, the vertex set of $\Gamma$ is $D^{(\Z)}$.
The edge set is $\left\{\overrightarrow{(f,g)}:\exists x_0:f(x)=g(x)\iff x\neq x_0\right\}$.
We disregard the labels and think of the graph as unlabeled.

Consider a finite set $F$ of elements of $\Z\wr D$.
Let $\widetilde{F}$ be the associated subgraph (see Definition~\ref{asssubgr}).
Remark that $|E(\widetilde{F})|=(d-1)|F|$.
A \textbf{leaf} we call a vertex which is included in exactly one edge of the subgraph and is at the tail of that edge.
The set of leaves we denote by $L(\widetilde{F})$.

We would like to estimate, for each vertex $f$ in $\widetilde{F}$, the number of points of the form $(k,f)$ in $\partial_{out}F$.
We will show that if $f$ is a leaf, there is at least one such point, and if $f$ is not a leaf, there are at least two.
The first case follows directly from the definition of a leaf.
Indeed, let $f$ be at the tail of $\overrightarrow{(f',f)}$.
Then for some $k_0$ we have $f(k_0)\neq f'(k_0)$ and $(k_0,f')\in F$.
Then for $i=(f'(k_0))^{-1}f(k_0)\in D$, we have $(k_0,f')\delta_i=(k_0,f)\in\partial_{out}F$.

Assume that $f$ is not a leaf.
Then either there are at least two edges ending in $f$, or there is at least one edge starting at $f$.
Once again, the first case follows directly from the definition of a leaf.
In the second case, there exists $k_0$ such that $(k_0,f)\in F$.
Let $a=\max\{k:(k,f)\in F\}$ and $b=\min\{k:(k,f)\in F\}$.
Then $(a+1,f)=(a,f)t$ and $(b-1,f)=(b,f)t^{-1}$ are in the outer boundary of $F$.

We obtain:

\begin{equation}\label{ineq}
|\partial_{out}F|\geq2(|V(\widetilde{F})|-|L(\widetilde{F})|)+|L(\widetilde{F})|=2|V(\widetilde{F})|-|L(\widetilde{F})|.
\end{equation}

Remark that if there are no leaves, the desired result follows.

\begin{lemma}\label{leaf}
Fix $n$.
Assume that
$$\min\left(\frac{2|V(G)|-|L(G)|}{|E(G)|}:G\mbox{ subgraph of }\Gamma\mbox{ with }|L(G)|\neq0\mbox{ and }|E(G)|\leq n\right)\leq1.$$
Then
\begin{equation*}
\begin{split}
&\min\left(\frac{2|V(G)|-|L(G)|}{|E(G)|}:G\mbox{ subgraph of }\Gamma\mbox{ with }|L(G)|\neq0\mbox{ and }|E(G)|\leq n\right)\\
&\geq\min\left(\frac{2|V(G)|}{|E(G)|}:G\mbox{ subgraph of }\Gamma\mbox{ with }|E(G)|\leq n-1\right).
\end{split}
\end{equation*}
\end{lemma}

\begin{proof}
We will prove the lemma by induction on $n$.
For small enough $n$, the assumed inequality is impossible, which gives us the base of the induction.

Assume that the statement is true for $n-1$ and consider $G$ with $n$ edges such that

$$\frac{2|V(G)|-|L(G)|}{|E(G)|}=\min\left(\frac{2|V(H)|-|L(H)|}{|E(H)|}:|L(H)|\neq0\mbox{ and }|E(H)|\leq n\right)\leq1.$$

If $G$ does not contain a leaf, the statement is trivial.
We now assume that $G$ contains a leaf.
Let $G'$ be the subgraph obtained by removing that leaf and the edge leading to it.
Then $|V(G')|=|V(G)|-1$ and $|E(G')|=n-1$.

Notice also that while new leaves may have appeared, at most one was removed, and thus $|L(G')|\geq|L(G)|-1$.
Therefore, as $\frac{2|V(G)|-|L(G)|}{|E(G)|}\leq1$, we have:
$$\frac{2|V(G')|-|L(G')|}{|E(G')|}\leq\frac{2|V(G)|-|L(G)|-1}{|E(G)|-1}\leq\frac{2|V(G)|-|L(G)|}{|E(G)|}.$$

It is in particular less than $1$, and we obtain:
$$\frac{2|V(G')|-|L(G')|}{|E(G')|}\geq\min\left(\frac{2|V(H)|}{|E(H)|}:H\mbox{ subgraph of }\Gamma\mbox{ with }|E(H)|\leq n-1\right),$$
either from the induction hypothesis or directly, depending on whether $G'$ has a leaf or not.
\end{proof}
Remark that $\frac{2|V(\widetilde{F})|-|L(\widetilde{F})|}{|E(\widetilde{F})|}\leq\frac{(d-1)|\partial_{out}F|}{|F|}\leq1$ by the assumption of the lemma.
The result of Lemma~\ref{main} then follows from equation (\ref{ineq}) and applying Lemma~\ref{leaf} for $n=|E(\widetilde{F})|$.
\end{proof}

Having proven Lemma~\ref{main}, we now need to show that the images of the standard sets $F_n$ minimize $\frac{2|V(G)|}{|E(G)|}$ over subgraphs of $\Gamma$ with a fixed amount of edges.
The associated subgraph of $F_n$ (see Definition~\ref{asssubgr}) is the subgraph of $\Gamma$ with vertices $\{f:\supp(f)\subset[\![1,n]\!]\}$ and all edges of $\Gamma$ between two of those vertices.
It has $d^n$ vertices and $n(d-1)d^n$ edges.

\begin{lemma}\label{inversecubal}
Let $G$ be a subgraph of $\Gamma$ with at most $n(d-1)d^n$ edges.
Then $\frac{2|V(G)|}{|E(G)|}\geq\frac{1}{n}$, and if they are equal, $G$ is a $d$-hypercube with double edges with $d^n$ vertices.
\end{lemma}

\begin{proof}
We will first see that we can assume that any directed edge is present simultaneously with its inverse.
Indeed, consider the set of directed edges $\overrightarrow{(f,g)}$ in $G$ such that $\overrightarrow{(g,f)}$ is not in $G$.
If there are at least two such edges, removing one and adding the inverse of the other changes neither $|V(G)|$ nor $|E(G)|$.
If there is exactly one such edge $\overrightarrow{(f,g)}$, then $|E(G)|$ is odd, and so it is at most $n(d-1)d^n-1$.
In that case by adding the edge $\overrightarrow{(g,f)}$ we obtain another subgraph with at most $n(d-1)d^n$ edges, and such that the quotient we consider is smaller.

We can therefore assume that any directed edge is present simultaneously with its inverse.
We replace every couple of edges $\overrightarrow{(f,g)}$ and $\overrightarrow{(g,f)}$ with one undirected edge $(f,g)$.
We will call the graphs we obtain $\overline{\Gamma}$ and $\overline{G}$.
We need to show that $\frac{|V(\overline{G})|}{|E(\overline{G})|}\geq\frac{1}{n}$, and if they are equal, $\overline{G}$ is a $d$-hypercube with $d^n$ vertices.
We prove a more general statement:

\begin{lemma}
Let $X$ be a subgraph of $\overline{\Gamma}$ with $V$ vertices. Then it has at most $\frac{d-1}{2}V\log_d(V)$ edges and equality is obtained only if $X$ is a $d$-hypercube.
\end{lemma}

\begin{proof}
We obtain that by induction on $d$ and then on $V$.

For $d=2$, remark that $\overline{\Gamma}$ is the bi-infinite hypercube.
As $X$ is finite, it is contained in some finite hypercube.
In that case, the question of maximizing the number of edges on a fixed number of vertices has already been answered in literature.
One proof is presented in Harper's book~\cite[Section~1.2.3]{Harper2004}.
Taking notations from the book, consider a subset $C\subset\Z$ of cardinal $c$, and a vertex $f$ of $\overline{\Gamma}$ with $\supp(f)\bigcap C=\emptyset$.
The set
$$\left\{g\in V(\overline{\Gamma}):f(x)=g(x)\forall x\notin C\right\}$$
is called a \textbf{$c$-subcube}.
A vertex set $S$ of cardinal $k=\sum_{i=1}^K2^{c_i}$ with $0\leq c_i<c_j$ for $i<j$ is \textbf{cubal} if it is a disjoint union of $c_i$-subcubes with the $c_i$-subcube being contained in the neighborhood of the $c_j$-subcube for all $i<j$.
Here, neighborhood means the set of points at distance $1$ in graph distance.
Remark that two cubal sets of the same cardinality are isomorphic.

By abuse of notation, for a set $S$ of vertices in $\Gamma$, we will denote by $E(S)$ the set of all edges between two vertices of $S$.
Then Theorem~$1.1$ from the cited section states
\begin{thm}[{\cite[Section~1.2.3~-~Theorem~1.1]{Harper2004}}]\label{thmharper}
	$S$ maximizes $|E(S)|$ over sets with cardinal $|S|$ if and only if $S$ is cubal.
\end{thm}

Computing the edges of a cubal set is not difficult, and this completes the case of $d=2$.
We now assume the result to be true for $d-1$ and will prove it for $d$ by induction on $V$.
The cases where $V\leq d-1$ are trivial.
We assume the result to be true for $|V(X)|\leq V-1$ and prove it for $V$.

Since $X$ is finite, without loss of generality we can assume its vertices are in $D^{(\N_0)}$, and that there are vertices of $X$ with different values in $0$.
We write the elements of $D$ as $g_i$, $i=1,2,\dots,d$.
Let $X_i$ be the subgraph of $X$ induced by the vertices $f$ with $f(0)=g_i$.
If one of them is empty, the proof of the induction step follows from the proof for $d-1$.
Therefore for the rest of this proof we assume that $X_i$ is non-empty for every $i$.
Let $V_i=|V(X_i)|$ and let $v_i=\frac{V_i}{V}$.
Remark that $\sum_{i=1}^dv_i=1$.
Then we have

$$|E(X)|\leq\sum_{i=1}^d|E(X_i)|+\sum_{1\leq i<j\leq d}\min(V_i,V_j).$$

By induction hypothesis, for all $i$:

$$|E(V_i)|\leq\frac{d-1}{2}V_i\log_d(V_i).$$

It is therefore sufficient to prove that
$$\frac{d-1}{2}V\log_d(V)\geq\sum_{i=1}^d\frac{d-1}{2}V_i\log_d(V_i)+\sum_{1\leq i<j\leq d}\min(V_i,V_j).$$

Remark that $\frac{d-1}{2}V\log_d(V)=\sum_{i=1}^d\frac{d-1}{2}V_i\log_d(V)$.
Denoting $C=\frac{d-1}{2\ln(d)}$ and dividing by $V$ we get the equivalent inequality:

$$\sum_{i=1}^dCv_i\ln(v_i)+\sum_{1\leq i<j\leq d}\min(v_i,v_j)\leq0.$$

Denote the left-hand side with $f(\bar{v})$ where $\bar{v}=(v_1,\dots,v_d)$.
The idea is to use the Lagrange Multipliers method to show that $f$ only finds a maximum when $v_1=v_2=\dots=\frac{1}{d}$ (recall the condition $\sum_{i=1}^dv_i=1$).
More precisely however, to avoid difficulties with the minimum function, we assume that $v_1\leq v_2\leq\dots\leq v_d$ and use Karush–Kuhn–Tucker~(KKT) conditions.
We obtain $f(\bar{v})=\sum_{i=1}^dCv_i\ln(v_i)+\sum_{i=1}^d(d-i)v_i$ and the Lagrangian function

$$L(\bar{v},\lambda,\bar{\mu})=f(\bar{v})+\lambda\left(\sum_{i=1}^dv_i-1\right)+\sum_{i=1}^{d-1}\mu_i(v_{i+1}-v_i).$$

Remark that we seek to maximize $f$ rather than minimize it, and the inequality $v_{i+1}-v_i$ is positive rather than negative - so the entire equation has the opposite sign of what one usually finds in KKT conditions.
We pose $v_0=\mu_0=\mu_d=0$ to simplify notation further in the proof.
Remark that $v_1\geq v_0$ and $\mu_0(v_1-v_0)=0$.
The necessary conditions we obtain are then, for $i=1,\dots,d$:
$$\frac{df}{dv_i}+\lambda+\mu_{i-1}-\mu_i=0,$$
$$\mu_i(v_{i+1}-v_i)=0,$$
and also $\mu_i\geq0$.
We still have $\sum_{i=1}^dv_i=1$ and $v_{i+1}\geq v_i$ for $i=0,\dots,d-1$.

We calculate
$$\frac{df}{dv_i}+\lambda+\mu_{i-1}-\mu_i=C(1+\ln v_i)+d-i+\mu_{i-1}-\mu_i+\lambda,$$
and thus
\begin{equation}\label{lambda}
-\lambda=C(1+\ln v_i)+d-i+\mu_{i-1}-\mu_i.
\end{equation}

Taking a sum over all $i$ and dividing by $d$ we obtain
$$-\lambda=C(1+\frac{1}{d}\sum_{i=1}^d\ln v_i)+\frac{d-1}{2}.$$

We replace that quantity in~\ref{lambda}:
$$C(1+\ln v_i)+d-i+\mu_{i-1}-\mu_i=C(1+\frac{1}{d}\sum_{i=1}^d\ln v_i)+\frac{d-1}{2}$$
and thus, expanding $C$:
$$\frac{d-1}{2d}\left((d-1)\log_dv_i-\sum_{j\neq i}v_j\right)=\mu_i-\mu_{i-1}+i-\frac{d+1}{2}.$$

Denote by $M$ the $d$ by $d$ matrix with $d-1$ on the diagonal and $-1$ on all other entries.
Let $u_i=\log_dv_i$ and $U=\bar{u}^\top$.
The equation can then be written
$$\frac{d-1}{2d}MU=\left(\mu_i-\mu_{i-1}+i-\frac{d+1}{2}\right)_i.$$

We now seek to understand the matrix $M$.
Consider a vector $Y\in\ker M$.
Then for all $i$, $dy_i-\sum y_i=0$, and so $y_i=y_j$ for all $i,j$.
Inversely, any vector of this type is in $\ker M$.
Therefore $\ker M$ is the vector space generated by the vector $(1)_i$.
It follows that the degree of the image of $M$ is $d-1$.

Consider now a vector $Y$ such that $\sum y_i=0$.
It is easy to check that $MY=dY$.
Therefore the image of $M$ is exactly the space of vectors with sum of the coordinates $0$.
Then if $MY=Y'$, we have that $Y=\frac{1}{d}Y'+\alpha(1)_i$ for some constant $\alpha$.

Assume that some fixed $\bar{v}$ is a local minimum.
It follows that there exists $\alpha$ such that for all $i$:
$$\frac{d-1}{2}\log_dv_i=\mu_i-\mu_{i-1}+i-\frac{d+1}{2}+\alpha.$$

Let $0=k_0<k_1<\dots<k_r=d$ be the sequence of indexes such that $\mu_i=0$ if and only if $i=k_j$ for some $j$.
Fix $0\leq j\leq r$.
As $\mu_i(v_{i+1}-v_i)=0$, we obtain that for all $k_j<i<k_{j+1}$, $v_{i+1}=v_i$.
Therefore for $k_j<i\leq k_{j+1}$, we have $v_i=v_{k_{j+1}}$.
Let $\gamma_j=\frac{d-1}{2}\log_dv_{k_{j+1}}-\alpha$.
Then for $k_j<i\leq k_{j+1}$:

$$\gamma_j=\mu_i-\mu_{i-1}+i-\frac{d+1}{2}.$$

We take the sum of those equations over $i$ with $k_j<i\leq k_{j+1}$ ($j$ is still fixed) and obtain:

$$(k_{j+1}-k_j)\gamma_j=\mu_{k_{j+1}}-\mu_{k_j}+\sum_{i=k_j+1}^{k_{j+1}}\left(i-\frac{d+1}{2}\right).$$

We have $\mu_{k_{j+1}}=\mu_{k_j}=0$ and $\sum_{i=k_j+1}^{k_{j+1}}i=\frac{1}{2}(k_{j+1}-k_j)(k_j+k_{j+1}+1)$ and thus
$$2\gamma_j=k_j+k_{j+1}-d.$$

Then for $k_j<i\leq k_{j+1}$:

$$\log_dv_i=\frac{2\gamma_j+2\alpha}{d-1}=\frac{k_j+k_{j+1}}{d-1}+\alpha'.$$

Denote $\tilde{d}=d^{\frac{1}{d-1}}$ and for $1\leq i\leq d$ let $j(i)$ be the unique number such that $k_{j(i)}<i\leq k_{j(i)+1}$.
Recall that $\sum_{i=1}^dv_i=1$.
Then
$$\sum_{i=1}^d\tilde{d}^{k_{j(i)}+k_{j(i)+1}}d^{\alpha'}=1,$$
so $d^{-\alpha'}=\sum_{i=1}^d\tilde{d}^{k_{j(i)}+k_{j(i)+1}}$ and for all $i=1,2,\dots,d$:
$$v_i=\frac{\tilde{d}^{k_{j(i)}+k_{j(i)+1}}}{\sum_{i'=1}^d\tilde{d}^{k_{j(i')}+k_{j(i')+1}}}.$$

Let $C'=\sum_{i=1}^d\tilde{d}^{k_{j(i)}+k_{j(i)+1}}$.
We calculate
$$C'f(\bar{v})=\frac{1}{2}\sum_{i=1}^d\tilde{d}^{k_{j(i)}+k_{j(i)+1}}(k_{j(i)}+k_{j(i)+1})-\frac{d-1}{2}\sum_{i=1}^d\tilde{d}^{k_{j(i)}+k_{j(i)+1}}\log_dC'+\sum_{i=1}^d\tilde{d}^{k_{j(i)}+k_{j(i)+1}}(d-i).$$

Splitting the sums by $j$, we have
$$\frac{1}{2}\sum_{i=1}^d\tilde{d}^{k_{j(i)}+k_{j(i)+1}}(k_{j(i)}+k_{j(i)+1})=\frac{1}{2}\sum_{j=0}^{r-1}\tilde{d}^{k_j+k_{j+1}}(k_{j+1}^2-k_j^2),$$
$$\frac{d-1}{2}\sum_{i=1}^d\tilde{d}^{k_{j(i)}+k_{j(i)+1}}\log_dC'=\frac{d-1}{2}\sum_{j=0}^{r-1}\tilde{d}^{k_j+k_{j+1}}(k_{j+1}-k_j)\log_dC'$$
and
$$\sum_{i=1}^d\tilde{d}^{k_{j(i)}+k_{j(i)+1}}(d-i)=\sum_{j=0}^{r-1}\tilde{d}^{k_j+k_{j+1}}\frac{(k_{j+1}-k_j)(2d-k_{j+1}-k_j-1)}{2}.$$

Then
$$C'f(\bar{v})=\frac{1}{2}\sum_{j=0}^{r-1}(k_{j+1}-k_j)\tilde{d}^{k_j+k_{j+1}}\left(2d-1-(d-1)\log_dC'\right).$$

Recall that we want to obtain $f(\bar{v})\leq0$.
It is thus equivalent to $\log_dC'\geq\frac{d-1}{2d-1}$.
We write:
$$C'=\sum_{i=1}^d\tilde{d}^{k_{j(i)}+k_{j(i)+1}}=\sum_{j=0}^{r-1}(k_{j+1}-k_j)\tilde{d}^{k_j+k_{j+1}}\geq(k_r-k_{r-1})\tilde{d}^{k_{r-1}+k_r}.$$

If all $v_i$ are equal (and thus $r=1$, $k_0=0$ and $k_1=d$), we obtain $\log_dC'=\frac{d-1}{2d-1}$.
Assume now that $r\geq2$.
We then have $x=k_{r-1}\geq1$ and $C'\geq(d-x)\tilde{d}^{d+x}$.
Deriving this function, we find a local maximum at $x=d-(d-1)\ln d$.
Therefore the minimal value over the given interval can be found at either $x=1$, or $x=d-1$, or both.
Calculating, we find $\log_dC'$ to be strictly larger than $\frac{d-1}{2d-1}$ in both of those points, which completes the proof.
\end{proof}
Lemma~\ref{inversecubal} follows.
\end{proof}

Theorem~\ref{thmlamp}$(1)$ follows from Lemma~\ref{main} and Lemma~\ref{inversecubal}.
By Lemma~\ref{equiv}, Theorem~\ref{thmlamp}$(2)$ follows from Theorem~\ref{thmlamp}$(1)$ (and the unicity of the optimal sets).

\section{Bounds for the Coulhon and Saloff-Coste inequality for the wreath product $\Z\wr D$}\label{csc-const-sect}

We will now prove Proposition~\ref{const}.
Recall its statement:
\begin{prop*}
The wreath product $Z\wr D$ verifies

$$C_{\Z\wr D,S_D}=\frac{\lim\frac{\ln\Fol(n)}{n}}{\lim\frac{\ln V(n)}{n}}=\frac{2\ln d}{\ln(\frac{1}{2}(1+\sqrt{4d-3}))}.$$
\end{prop*}

\begin{proof}
We have obtained in Theorem~\ref{thmmain} that $\lim\sqrt[n]{\Fol(n)}=d^2$.
It is therefore sufficient to calculate the exponent of its volume growth for $S_D$.

Consider an element $g$ of length $n$ in the group and let its support be $[m,p]$.
We can write $g=t^mAt^i$ where $A$ is a non-reducible word on $t$ and the different $\delta_i$ (without $t^{-1}$).
It is not hard to see that $A$ contains the letter $t$ exactly $p-m$ times.
Remark that any representation of $g$ in $S_D$ must contain $t$ or $t^{-1}$ at least $p-m$ times.
Additionally, it contains a letter of the form $\delta_i$ at least as many times as $A$ does.
Therefore $|A|\leq|g|=n$.

Notice also that $|m|\leq n$ and $|i|\leq2n$.
Therefore, if we denote by $V'(n)$ the amount of non-reducible words on $t$ and the different $\delta_i$ of length at most $n$, we obtain
$$V'(n)8n^3\geq V(n)\geq V'(n).$$

We now seek to calculate $V'(n)$.
Notice that the only condition on the words counted by this function is to not have two consecutive letters of the form $\delta_i$.
It is not difficult to describe a recurrence relation.
If a word ends in $t$, by removing it we obtain a word of length at most $n-1$.
If it ends in $\delta_i$ for some $i$, the previous letter must be $t$.
By removing both, we obtain a word of length at most $n-2$.
Therefore
$$V'(n)=V'(n-1)+(d-1)V'(n-2),$$
and by calculating the series we obtain
$$\lim\sqrt[n]{V(n)}=\lim\sqrt[n]{V'(n)}=\frac{1+\sqrt{4d-3}}{2}.$$
\end{proof}

It is worth noting that the exact value of the volume growth power series $\sum_nV(n)x^n$ for the standard generating set has been described by Parry~\cite{Parry1992a}.

We then prove Remark~\ref{const2}.
It states:
\begin{remark*}
For the switch-walk-switch generating set $S_{sws}=\{t,\delta,t\delta,\delta t,\delta t\delta\}$ on the lamplighter group $\Z\wr\Z/2\Z$, we have

$$C_{\Z\wr\Z/2\Z,S_{sws}}=\frac{\liminf\frac{\ln\Fol_{sws}(n)}{n}}{\lim\frac{\ln V_{sws}(n)}{n}}\leq2.$$
\end{remark*}
\begin{proof}
The standard F{\o}lner sets verify that $\frac{|\partial'_{out}F_n|}{|F_n|}=\frac{2}{n}$ where by $\partial'_{out}$ we denote the outer boundary with regards to $S_{sws}$.
Therefore $\lim\sqrt[n]{\Fol_{sws}(n)}\leq4$.
Similarly, we now need to calculate the exponent of its volume growth.

Equivalently to the proof of Proposition~\ref{const}, the switch-walk-switch volume is also controlled by a polynomial times the volume growth assuming no multiplication by $t^{-1}$.
It is easy to check that the latter is $n\mapsto2^n$.
Thus $8n^32^n\geq V_{sws}(n)\geq2^n$ and $\lim\sqrt[n]{V_{sws}(n)}=2$.
\end{proof}

\section{The Baumslag-Solitar group $BS(1,2)$}\label{bssect}

We now prove Theorem~\ref{bsthm}.
Remark that for the edge boundary, the contributions by each element of the generating set are always disjoint.
Recall that the vertex set of the associated graph $\Gamma'$ is the dyadic numbers.
For the standard generating set, the edges are the couples $\overrightarrow{(f,g)}$ such that $|f-g|$ is a power of $2$, and are labeled $b$.
For any finite subgraph of $\Gamma'$, up to translation we can assume that its vertices are all integers, in which case edges will be defined by positive powers of $2$.

We introduce the following notations.
Consider a subset $F\subset BS(1,2)$ and let $\widetilde{F}$ be the associated subgraph.
The set of edges $\overrightarrow{(f,g)}\in E(\widetilde{F})$ such that $\overrightarrow{(g,f)}\notin E(\widetilde{F})$ we will denote $L(\widetilde{F})$ (remark that unlike Section~\ref{mainsection}, this is a set of edges rather than vertices).
Let $\overline{\Gamma'}$ be the graph obtained by replacing every couple of edges $\overrightarrow{(f,g)}$ and $\overrightarrow{(g,f)}$ in $\Gamma'$ with one undirected edge.
Let $\overline{\widetilde{F}}$ be the subgraph of $\overline{\Gamma'}$ obtained by taking the images of $E(\widetilde{F})\setminus L(\widetilde{F})$ and all adjacent vertices.

For any $i\in\Z$, consider the transformation $x\mapsto x+2^i$ which is well defined on the vertices of $\overline{\Gamma'}$.
Denote by $o(\overline{\widetilde{F}})$ the total amount of distinct non-trivial orbits in $V(\overline{\widetilde{F}})$ over all $i$.
Remark that this quantity is well defined for any subgraph $\overline{\widetilde{F}}$ of $\overline{\Gamma'}$.

\begin{lemma}\label{comparebs}
Consider a subset $F\subset BS(1,2)$ with $\frac{|\partial_{edge}F|}{|F|}\leq1$.
Then

$$|F|\geq|E(\overline{\widetilde{F}})|+o(\overline{\widetilde{F}})$$
and
$$\frac{|\partial_{edge}F|}{|F|}\geq2\frac{|V(\overline{\widetilde{F}})|+o(\overline{\widetilde{F}})}{|E(\overline{\widetilde{F}})|+o(\overline{\widetilde{F}})}.$$
\end{lemma}

\begin{proof}
We start by estimating $|F|$.
By definition of associated subgraph, we have
$$2|F|=|E(\widetilde{F})|.$$
By definition of $L(\widetilde{F})$ and $\overline{\widetilde{F}}$ we have
$$|E(\widetilde{F})=|L(\widetilde{F})|+2|E(\overline{\widetilde{F}})|.$$
We now need to prove that $|L(\widetilde{F})|\geq2o(\overline{\widetilde{F}})$.
Indeed, consider any non-trivial orbit and let $f$ and $g$ be respectively its smallest and largest elements.
If the step of the orbit is $2^i$, we have that $(i,f)$ and $(i,g)$ are elements of $F$.
Then $\overrightarrow{(f-2^i,f)}$ and $\overrightarrow{(g,g+2^i)}$ must both be edges in $\widetilde{F}$.
As $f-2^i$ and $g+2^i$ are not in $\overline{\widetilde{F}}$, those two edges are then in $L(\widetilde{F})$.

We now estimate the boundary.
We will denote by $\partial_{edge}^aF$ and $\partial_{edge}^bF$ the edges in $\partial_{edge}F$ labeled $a$ and $b$ respectively.
It follows directly from the definition of associated graph and $L(\widetilde{F})$ that
$$|\partial_{edge}^bF|=|L(\widetilde{F})|.$$

Consider a vertex $f\in V(\overline{\widetilde{F}})$.
By definition of $\overline{\widetilde{F}}$, there is $g$ such that $\overrightarrow{(f,g)}\in E(\widetilde{F})$.
It follows that there is $x$ such that $(x,f)\in F$.
There are thus at least two elements of $\partial_{edge}^aF$ such that their corresponding configuration is $f$.
We obtain:
$$|\partial_{edge}^aF|\geq2|V(\overline{\widetilde{F}})|.$$
Then:
$$\frac{|\partial_{edge}F|}{|F|}\geq2\frac{|V(\overline{\widetilde{F}})|+\frac{1}{2}|L(\widetilde{F})|}{|E(\overline{\widetilde{F}})|+\frac{1}{2}|L(\widetilde{F})|}\geq2\frac{|V(\overline{\widetilde{F}})|+o(\overline{\widetilde{F}})}{|E(\overline{\widetilde{F}})|+o(\overline{\widetilde{F}})}.$$
\end{proof}

We now seek to understand subgraphs of $\overline{\Gamma'}$ that optimize this quotient.
We claim that the subgraphs with vertices $[\![1,n]\!]$ do so.
We will denote $e(n)$ (respectively $o(n)$) the amount of edges (respectively orbits) in the subgraph with vertices $[\![1,n]\!]$ and all induced edges.
We prove results similar to Theorem~\ref{thmharper}, which we cited from Haprer's book~\cite{Harper2004}.

\begin{lemma}\label{baums}
Let $G$ be a subgraph of $\overline{\Gamma'}$ with $n$ vertices.
Then

$$|E(G)|\leq e(n).$$
\end{lemma}

\begin{proof}
We start by estimating $e(n)$.
For $n\in\N$ with $2^{k-1}<n\leq2^k$ we have that $e(n)$ is the sum over $i$ of the amount of couples of elements of $[\![1,n]\!]$ with difference $2^i$.
In other words,
$$e(n)=\sum_{i=0}^{k-1}(n-2^i)=kn-2^k+1.$$

We will now prove the result of the lemma by induction on $n$.
The base $n=1$ is trivial.
Fix a subgraph $G$ with $n$ vertices.
Up to translation by an integer we can assume that the vertices of $G$ are integer and the smallest is $1$.

Assume first that all vertices of $G$ are odd integers.
Let $i$ be the largest integer such that $2^i$ divides all elements of $\{f-1:f\in V(G)\}$.
Define $\psi(f)=\frac{f-1}{2^i}+1$.
Then the set $\psi(V(G))$ is a set of integers with the smallest element being $1$.
Furthermore, for any $f$ and $f'$ divisible by $2^i$, $|f-f'|$ is a power of $2$ if and only if $|\psi(f)-\psi(f')|$ is.
Without loss of generality we can replace $G$ by the subgraph of $\overline{\Gamma'}$ with vertices $\psi(V(G))$.

We can therefore assume $G$ has both even and odd elements.
Let $G_1$ be the subgraph induced by vertices of $G$ that are odd integers, and $G_2$ the subgraph induced by even integers.
Denote the cardinal of their vertex sets by $n_1$ and $n_2$ respectively.
Notice that $n=n_1+n_2$.

As the difference (in terms of integer values) between vertices of $G_1$ and $G_2$ is always odd, there can be an edge if and only if the difference is $1$.
Then by induction hypothesis the number of edges in $F$ is not greater than
$$e(n_1)+e(n_2)+2\min(n_1,n_2)-\varepsilon$$
where $\varepsilon=1$ if $n_1$ and $n_2$ are equal, and $0$ otherwise.

Let us denote $2^{k_1-1}<n_1\leq2^{k_1}$, $2^{k_2-1}<n_2\leq2^{k_2}$.
Without loss of generality, assume $n_1\leq n_2$.
Then $k-1\leq k_2\leq k$.
Let $\delta=k-k_2$.
We calculate

\begin{align*}
& |E(G)|-e(n)\leq e(n_1)+e(n_2)+2n_1-\varepsilon-e(n) \\
& =k_1n_1-2^{k_1}+1+k_2n_2-2^{k_2}+1+2n_1-\varepsilon-k(n_1+n_2)+2^k-1 \\
& =(k_1-k+2)n_1-\delta n_2+1+2^k-2^{k_1}-2^{k_2}-\varepsilon=A.
\end{align*}

We seek to prove that $A\leq0$.
We have $\delta=0$ or $1$.
First, assume $\delta=0$.
Then $2^k=2^{k_2}$ and
$$A=(k_1+2-k)n_1+1-2^{k_1}-\varepsilon.$$
As $k_1\leq k-1$, we have $A\leq n_1+1-2^{k_1}-\varepsilon\leq 0$.

Assume now $\delta=1$.
Then
$$A=(k_1+2-k)n_1-n_2+1+2^{k-1}-2^{k_1}-\varepsilon.$$

Assume first that $k_1\leq k-2$.
Then $A\leq 2^{k_2}-n_2+1-\varepsilon\leq0$.

As $k_1\leq k_2=k-1$, the only case left is $k_1=k-1$.
Then
$$A=n_1-n_2+1-\varepsilon.$$
If $n_1-n_2=0$, then $\varepsilon=1$ and $A=0$.
If $n_1-n_2\leq-1$, then $A\leq-\varepsilon\leq0$.
\end{proof}

\begin{lemma}\label{baums2}
Let $G$ be a subgraph of $\overline{\Gamma'}$ such that $|V(G)|+o(G)=2n-1$ or $2n$.
Then

$$|E(G)|\leq e(n).$$
\end{lemma}

Remark that $o(n)=n-1$.
Indeed, for each non-trivial orbit, consider the second smallest (in terms of the natural order on the integers) element of the orbit.
If it has step $2^i$, that element is between $1+2^i$ and $2^{i+1}$.
Therefore these elements are disjoint, and cover every element in $[\![1,n]\!]$ except the element $1$.

\begin{proof}
We will now prove the result of the lemma by induction on $n$.
The base $n=1$ is trivial.
Fix a subgraph $G$ such that $|V(G)|+o(F)=2n-1$ or $2n$.
Up to translation by an integer we can assume that the vertices of $G$ are integer and the smallest is $1$.
As in Lemma~\ref{baums}, we can assume that $G$ has even vertices, and split it into $G_1$ and $G_2$ based on the parity of the vertices.

Let $n_1=\lceil\frac{|V(G_1)|+o(G_1)}{2}\rceil$ and $n_2=\lceil\frac{|V(G_2)|+o(G_2)}{2}\rceil$.
Let $2^{k_1-1}<n_1\leq2^k_1$ and $2^{k_2-1}<n_2\leq2^k_2$.
Without loss of generality, let $n_1\leq n_2$.

Consider first the case where there is at least one edge between $G_1$ and $G_2$.
Then $o(G)\geq o(G_1)+o(G_2)+1$ and $2n-1\geq|V(G)|+o(G)-1\geq|V(G_1)|+o(G_1)+|V(G_2)|+o(G_2)$.
This implies that $n_1+n_2\leq n$.
The result then follows from the proof of Lemma~\ref{baums}.

Consider now the case where there is no edge between $G_1$ and $G_2$.
Then $o(G)\geq o(G_1)+o(G_2)$ and $2n\geq|V(G)|+o(G)\geq|V(G_1)|+o(G_1)+|V(G_2)|+o(G_2)$.
This implies that $n_1+n_2\leq n+1$.
By induction hypothesis we have
$$|E(G)|=|E(G_1)|+|E(G_2)|\leq e(n_1)+e(n_2).$$
We calculate
$$e(n_1)+e(n_2)-e(n)=k_1n_1+k_2n_2-kn+2^k-2^{k_1}-2^{k_2}+1=A.$$
Assume first that $k_2=k$.
Then
$$A\leq(k_1-k)n_1-2^{k_1}+1+k=(k_1-k)(n_1-1)-2^{k_1}+k_1+1\leq0$$
as $n_1\geq1$.

Assume now that $k_2\leq k-2$.
However, in that case $n_1+n_2\leq2^{k-1}<n$, and the result follows from the proof of Lemma~\ref{baums}.

Assume that $k_2=k-1$.
Then
\begin{equation*}
\begin{split}
A&=(k_1+1-k)n_1-n+2^{k-1}-2^{k_1}+k\\
&=(k_1+1-k)(n_1-1)-(n-2^{k-1})-(2^{k_1}-k_1-1)\leq0,
\end{split}
\end{equation*}
as $n_1\geq1$ and $k_1\leq k_2=k-1$.
\end{proof}

Combining those two lemma, we obtain:

\begin{cor}
Let $G$ be a subgraph of $\overline{\Gamma'}$ such that $|V(G)|+o(G)\leq2n$ and $\frac{|V(G)|+o(G)}{|E(G)|+o(G)}\leq1$.
Then

$$\frac{|V(G)|+o(G)}{|E(G)|+o(G)}\geq\frac{n+o(n)}{e(n)+o(n)}.$$
\end{cor}

Remark that $n+o(n)=2n-1$.

\begin{proof}
If $o(G)\leq o(n)$, the result follows trivially from Lemma~\ref{baums2}.
Assume then that $o(G)>o(n)=n-1$.

Thus $|V(G)|=n'\leq n$.
From Lemma~\ref{baums} we obtain $|E(G)|\leq e(n')$.
We have:

$$\frac{|V(G)|+o(G)}{|E(G)|+o(G)}\geq\frac{n'+o(n')+(o(G)-o(n'))}{e(n')+o(n')+(o(G)-o(n'))}\geq\frac{n'+o(n')}{e(n')+o(n')}\geq\frac{n+o(n)}{e(n)+o(n)}.$$
\end{proof}

The large inequality of Theorem~\ref{bsthm} follows directly from this corollary and Lemma~\ref{comparebs}.
The strict inequality is obtained by noticing that if $|V(\overline{\widetilde{F}})|+o(\overline{\widetilde{F}})<n+o(n)$, then $|V(\overline{\widetilde{F}})|+o(\overline{\widetilde{F}})\leq2n-2$ and we can apply the corollary for $n-1$.
This concludes the proof of Theorem~\ref{bsthm}.

Finally, we consider $BS(1,p)$ with the standard generating set.
\begin{ex}\label{exbsp}
There exist natural numbers $p$ and $n$ and a set $F\subset BS(1,p)$ such that $|F|\leq|F_n|$ (where $F_n=\{(k,f):k\in[\![0,n-1]\!],f\in\Z,0\leq f<p^n\}$) and

$$\frac{|\partial_{edge}F|}{|F|}<\frac{|\partial_{edge}F_n|}{|F_n|}\leq1.$$
In particular, this is true for $n=3$ and $p$ divisible by $3$ with $36\leq p\leq60$.
\end{ex}

\begin{proof}
The standard F{\o}lner set $F_n$ has $np^n$ elements.
We calculate $|\partial_{edge}^aF_n|=2p^n$ and $|\partial_{edge}^bF_n|=2\frac{p^n-1}{p-1}$.
Then
$$|\partial_{edge}F_n|=2p^n\frac{p-\frac{1}{p^n}}{p-1}\leq2p^n\frac{p}{p-1},$$
and $\frac{|\partial_{edge}F_n|}{|F_n|}\leq\frac{2}{n}\frac{p}{p-1}$.
Therefore for $n=3$ and $p\geq3$ we have $\frac{|\partial_{edge}F_3|}{|F_3|}\leq1$.
Assume that $p$ is divisible by $3$ and let
$$F=\{(k,f):k\in[\![0,3]\!],f=\sum_{i=0}^3\varepsilon_ip^i\mbox{ with }0\leq\varepsilon_i<\frac{p}{3}\}.$$
Then $|F|=4(\frac{p}{3})^4=3p^3\frac{4p}{243}$ and for $p\leq60$ we obtain $|F|\leq|F_3|$.
Furthermore, $|\partial_{edge}^aF|=2(\frac{p}{3})^4$ and $|\partial_{edge}^bF|=8(\frac{p}{3})^3$.
Thus
$$|\partial_{edge}F|=4\left(\frac{p}{3}\right)^4\left(\frac{1}{2}+\frac{6}{p}\right),$$
and for $p\geq36$ we have:
$$\frac{|\partial_{edge}F|}{|F|}=\frac{1}{2}+\frac{6}{p}\leq\frac{2}{3}<\frac{2}{3}\frac{p-\frac{1}{p^3}}{p-1}=\frac{|\partial_{edge}F_3|}{|F_3|}.$$
\end{proof}
Therefore the result of Theorem~\ref{bsthm} is not true for $BS(1,p)$ for all $p$.

\bibliographystyle{plainurl}

\end{document}